\documentclass[a4paper,12pt,reqno,superscriptaddress,showkeys,nofootinbib]{revtex4}
\usepackage[centertags]{amsmath}
\usepackage{amsfonts}
\usepackage{amssymb}
\usepackage{amsthm}
\usepackage{newlfont}
\usepackage{stmaryrd}
\usepackage{mathrsfs}
\usepackage{euscript}
\usepackage{siunitx}
\usepackage{physics}
\usepackage{algorithm2e}
\usepackage{graphicx,subcaption}
\usepackage{enumitem}
\usepackage{natbib} 
\usepackage{braket}
\usepackage{bbm}
\usepackage{graphicx}
\usepackage{color}
\usepackage{floatrow}
\usepackage{caption}
\usepackage{hyperref}
\usepackage{hhline}
\usepackage{hyperref}
\usepackage{cleveref}
\usepackage{amsmath}
\usepackage[bb=dsserif]{mathalpha}
\usepackage{bm}

\theoremstyle{plain}
  \newtheorem{theorem}{Theorem}[section]
  \newtheorem{corollary}[theorem]{Corollary}
  \newtheorem{proposition}[theorem]{Proposition}
  
\theoremstyle{definition}

\theoremstyle{remark}

\numberwithin{equation}{section}

\newcommand{\bbP}{{\mathbb P}}

\newcommand{\bbR}{{\mathbb R}}

\newcommand{\opunit}{\text{1}\kern-0.22em\text{l}}

\newcommand{\frm}{{\mathfrak m}}

\newcommand{\frS}{{\mathfrak S}}

\DeclareMathAlphabet{\mathpzc}{OT1}{pzc}{m}{it}

\newcommand{\fig}{Fig.\;}

\newcommand{\rel}{\,|\,}

\newcommand{\id}{\textrm{d}}

\usepackage{floatrow}
\usepackage{caption}

\begin{document}
\title{On the Poisson equation for nonreversible Markov jump processes}
\author{Faezeh Khodabandehlou} \affiliation{Department of Physics and Astronomy, KU Leuven, Belgium} \author{Christian Maes} \affiliation{Department of Physics and Astronomy, KU Leuven, Belgium}	\email{faezeh.khodabandehlou@kuleuven.be}
\author{Karel Neto\v{c}n\'{y}}
\affiliation{Institute of Physics, Czech Academy of Sciences, Prague, Czech Republic}

\begin{abstract}
We study the solution $V$ of the Poisson equation $LV + f=0$ where $L$ is the backward generator of an irreducible (finite) Markov jump process and $f$ is a given centered state function.  Bounds on $V$ are obtained using a graphical representation derived from the Matrix Forest Theorem and using a relation with mean first-passage times.  Applications include estimating time-accumulated differences during relaxation toward a steady nonequilibrium regime.
\end{abstract}
\maketitle

\section{Introduction}
A fascinating part of mathematical physics concerns the connection between certain (partial) differential equations and stochastic processes, \cite{Doob}.  Famous examples include the heat equation which is associated to Brownian motion, and the telegraph equation which connects with run-and-tumble dynamics.  As a general technique, Feynman-Kac formul{\ae} provide a relation between the (semi)group kernel generated by a (quantum) Hamiltonian and a stochastic representation, aka path-integral, allowing perturbative analysis and diagrammatic developments. Summing over walks or paths and their higher-dimensional versions is indeed a typical subject of stochastic geometry.  The associated {\it arborification}, {\it i.e.}, the possible restriction of those sums to (spanning sets of) trees is often a major simplification.  That is the subject of the Matrix Tree and Matrix Forest Theorems, which, while mainly in linear algebra, has things to offer in stochastic analysis as well.  That is also a main theme of the present paper, viz., to give a graphical representation in terms of trees of the solution of Poisson equations in the context of Markov jump processes.\\

A second and  related subject is to connect solutions of the Poisson equation with mean first-passage times.  The study of first-passage times comes with a useful intuition, and their behavior is obviously strongly tied to graph properties, \cite{redner}. In the end, that connection with mean first-passage times combined with graphical representations  allows to give estimates on the solution of the Poisson equation.  We refer to \cite{posap1, posap2, posap3} for the case of diffusions.  Such bounds are relevant for a number of applications, including the low-temperature properties of excess heat and of relaxational behavior, and all that within the context of finite Markov jump processes.  In that respect, we emphasize that the techniques of the present paper apply outside the traditional context of reversibility, and indeed are motivated by problems in nonequilibrium statistical mechanics.\\

 In the next section, the  setup within the framework of the  Markov jump process  is introduced.\\
  The Poisson equation is presented, in various forms, in Section \ref{poisson}, where we focus on relations between different concepts relating quasipotentials with mean first-passage times.\\
 In Section \ref{graph}, we give the graphical representation of solutions of the various Poisson equations.  The main graphical representations are formulated as  Theorem~\ref{th1} and Theorem~\ref{th2}.\\  Section \ref{bou} is an application of these relations and graphical representations. It gives bounds on the solutions of the Poisson equation.  We mention there an application for nonequilibrium thermal physics when the Markov jump process depends on parameters such as the inverse temperature.  We have three main estimates presented in inequalities \eqref{bb}, \eqref{bf}, and \eqref{bound}.\\
The Appendix presents the broader context of the graphical representations and in particular Appendix \ref{broa} enters into more mathematical details related to the Matrix Forest Theorem. 

\section{Setup}\label{mar}
Given a finite set $K$ of states $x,y,z,\ldots \in K$, we consider a Markov jump process $X_t\in K$
with transition rates $k(x,y), x\neq y$, for the jump $x\rightarrow y$. We speak about $X_t$ as the position of a random walker at time $t\geq 0$, but obviously, from a physics perspective, the states $x\in K$ do not need to model positions but can correspond to many-body configurations.\\
It naturally gives rise to a finite digraph, {\it i.e.}, an ordered pair of sets  $G =(\cal V, \cal E)$, where the vertex set $\cal V = K$ and $\cal E$ is the set of ordered pairs (arcs) $(x,y)$ of vertices $x,y\in V$ for which $k(x,y) >0$.  We assume  that the resulting graph $G$ is irreducible (strongly connected).  The latter implies the exponentially fast convergence, for time $t\uparrow \infty$,
\begin{equation}\label{cu}
\langle h(X_t)\rangle_x := \langle h(X_t) \,|\, X_0=x\rangle \longrightarrow \langle h\rangle^s =: h^s
\end{equation}
of expectations in the Markov process for every real-valued function $h$ on $K$ and independent of the initial condition $x$, toward the stationary expectation
\[
h^s =\sum_x h(x)\rho^s(x), \qquad \rho^sL =0
\]
for the unique stationary probability distribution $\rho^s>0$, and with $L$ the backward generator having matrix elements
\begin{align}\label{Lel}
 L_{x,y} & =k(x,y),\quad x\neq y\notag \\
 L_{x,x}&=-\sum_y k(x,y)
 \end{align}
In other words,
\begin{align}\label{Lf}
Lh\, (x)=\sum_y k(x,y)\,[h(y)-h(x)]
 \end{align}
gives rise to the semigroup $e^{L \,t}, t \geq 0$, with
 \begin{align}\label{Lsemi}
e^{L \,t} h(x)=\left \langle h(X_t)|X_0=x\right \rangle,\qquad \langle e^{tL}h\rangle^s = \langle h\rangle^s 
 \end{align}
as appears in \eqref{cu} as well.

 \section{Poisson equation}\label{poisson}
 By {\it Poisson equation} (in the above setup) we mean in general an equation for a real-valued function $U$ on $K$, solving
 \begin{equation}\label{genp}
     LU(x) + f(x)=0, \quad x\in H;\qquad U(x) = g(x),\quad x\notin H
 \end{equation}
 for a given subset $H\subset K$ and given functions $f: H \rightarrow \bbR,\,\, g: K\setminus H \rightarrow \bbR$.
 
 \subsection{Quasipotential}
 As a special but important case, we can consider $H =K$, where we deal with the Poisson equation
\begin{equation}\label{prop}
     LV\,(x) + f(x) =0 ,\quad  x\in K
 \end{equation}
 writing then $V$ instead of $U$, and we call $V$ the quasipotential with given source $f$, where we must require $\langle f\rangle^s = f^s = 0$ to have a solution. The solution to \eqref{prop} is unique up to an additive constant.  If we require that $\langle V\rangle^s=0$, then, clearly,
 \begin{equation}\label{exc}
 V(x) = \int_0^\infty \id t\,e^{tL}f\,(x), \quad x\in K
\end{equation}
which can be viewed as the accumulated excess of \eqref{cu}, in the sense that,
\begin{equation}\label{vint}
V(x) = \int_0^\infty \id t\,[ \,\langle h(X_t) \,|\, X_0=x\rangle - \langle h\rangle^s\,]  
\end{equation}
for $f=  h - \langle h\rangle^s$.\\
In the case where $f = LE$ for some potential function $E$, we have $V = E - \langle E\rangle^s$, which is one reason to call, more generally, the solution \eqref{vint} $V$ of \eqref{prop} a quasipotential. Obviously, the physical interpretation of $V$ in \eqref{vint} strongly depends on the meaning of $h$.  In the case where $h(x)$ is the expected instantaneous heat flux when the state is $x$, the quasipotential is closely related to the so-called excess heat; see \cite{jchemphys,mathnernst}.

\subsection{Stopped accumulation}\label{sae}
As an immediate generalization and for nonempty $H$, we introduce the random escape time 
\[
T_H := \inf\{t\geq 0, X_t \notin H\},\qquad T_H=0 \,\text{ if  }\; X_0 \notin H
\]
and consider the expected accumulation till that (random) stopping time:
\begin{eqnarray}\label{accum}
V_H(x) = \left\langle \phi_H \,|\,X_0=x\right\rangle; \qquad \qquad \phi_H(x):=\int_0^{T_H} \id t\,f(X_t)
\end{eqnarray}
The function $V_H$ is the unique solution of the Poisson equation
\begin{eqnarray}\label{tauH}
     LV_H(x) + f(x) = 0, \quad x\in H; \qquad \quad  V_H(x) = 0, \quad x\notin H
 \end{eqnarray}
 That reduces to \eqref{prop} when $H=K$.\\
 For completeness, we give the proof of \eqref{tauH}.
 \begin{proof}[Proof of \eqref{tauH}]
Define the stopped process which quits running after escaping from $H$:
\begin{equation}
	X_H(t) = 
	\begin{cases}
		X(t) & \text{for } t \leq T_H
		\\
		X(T_H) & \text{for } t > T_H
	\end{cases}
\end{equation}
It is again a Markov process with modified transition rates
\begin{equation}\label{rate-mod}
	k_H(x,y) = 
	\begin{cases}
		k(x,y) & \text{for } x \in H,\, y \in K
		\\
		0 & \text{for } x \not\in H
	\end{cases}
\end{equation}
and the corresponding generator is denoted by $L_H$. That process is not ergodic and its (in general nonunique) stationary distributions have support in 
$K \setminus H\neq \emptyset$. 
We have
\begin{equation}
\begin{split}
\bbP(T_H > t \rel X(0) = x) &=
\bbP(X_H(t) \in H \rel X_H(0) = x)
\\
&= (e^{t L_H} 1_H)(x)
\end{split}
\end{equation}
By irreducibility of the original process, the largest eigenvalue of $L_H$ (which is real by the Perron-Frobenius theorem) is strictly negative, and hence 
$T_H$ has exponentially tight distribution, 
$\bbP(T_H > t) \leq e^{-c t}$ for some $c > 0$ and $t$ large enough, uniformly with respect to the initial condition.  It means that any distribution in $H$ will vanish in the limit $t\uparrow \infty$.  As a result, the function $f_H(x) = f(x) 1_H(x)$ (where $1_H$ is the indicator for the set $H$) can be integrated to obtain \eqref{accum}:
\[
V_H(x) = \int_0^\infty e^{tL_H}f_H(x)\,\id t
\]
or 
\[
L_HV_H(x) = -f_H(x)
\]
which is \eqref{tauH}.
\end{proof}

\subsection{First-passage times}
As a special case we take $f=1$ in the above \eqref{accum}, to introduce the mean escape time $\frS_H(x)$ from $H$ when started from $x$:
\begin{equation}
    \frS_H(x) = \langle T_H\,|\,X_0=x\rangle
\end{equation}
It satisfies
\begin{eqnarray}\label{taupH}
     L\,\frS_H(x) + 1 = 0, \quad x\in H;\qquad\quad  \frS_H(x) = 0, \quad x\notin H
 \end{eqnarray}
For $x \not\in H$ we have
\begin{equation}
	L\frS_H\, (x) = \sum_{y \in H} k(x,y) \frS_H(y)
\end{equation}
By combining with~\eqref{taupH} and using the stationarity $\langle L \,\frS_H \rangle^s = 0$, we get
\begin{equation}\label{sum-rule}
	\begin{split}
		\sum_{x \not\in H} \sum_{y \in H }\rho^s(x) k(x,y) \frS_H(y)
		= \rho^s(H) \,\;(=\text{Prob}^s[x\in H])
	\end{split}
\end{equation}
which relates the escape times from $H$ when starting from (inner boundary) states in 
$H$, with the expected stationary number of jumps to $H$ from outside (outer boundary). \\

As a special case, we fix a state $z\in K$ and take $H = K\setminus\{z\}$. We put $\tau(x,z) := \frS_{H}(x)$ and then from \eqref{taupH},
\begin{equation}\label{taup}
     \sum_y k(x,y)[\tau(y,z) - \tau(x,z)] +1 =0, \quad x\neq z, \qquad \tau(z,z) = 0
 \end{equation}
which is the Poisson equation characterizing the\textit{ mean first-passage} time to reach $z$ when started from some $x\in K$.\\
 
 There is a useful relation between the quasipotential, solution of the Poisson equation \eqref{prop}, and the mean first-passage times:
  \begin{proposition}\label{pro1}
\begin{equation}\label{gr}
    V(x) = -\sum_z \rho^s(z)\, f(z)\,\tau(x,z) + \text{constant}
\end{equation}
 where the additive constant is fixed, if needed, by the condition $\langle V\rangle^s=0$.\\
 \end{proposition}
 As a consequence,
 \begin{equation}\label{difv}
    V(x)-V(y)= \sum_z \, \rho^s(z) \, f(z)(\tau (y,z) -\tau (x,z) )
\end{equation}
  \begin{proof}[Proof of Proposition \ref{pro1}]
We show that $V$ in \eqref{gr} is satisfying the Poisson equation \eqref{prop}. Keep $H = K\setminus\{z\}$; the identity \eqref{sum-rule} reads
\begin{equation}\label{hit-sum}
	\begin{split}
		\rho^s(z) \sum_x k(z,x) \,\frS_{H}(x) = 1 - \rho^s(z)
	\end{split}
\end{equation}
Therefore,
 \begin{equation}
     L\,\frS_{H}(x) = g_z(x) - \langle g_z\rangle^s,\qquad\text{ for } \,\; g_z(x) := \frac{\delta_{x,z}}{\rho^s(z)}
 \end{equation}
 valid for all $x\in K$, which means that the $\tau(x,z)$'s of \eqref{taup} play the role of Green functions for the Poisson equation.  In particular, since always $f(x) = \sum_z\rho^s(z) f(z) g_z(x)$, we obtain \eqref{gr} from the following calculation
 \begin{equation}
	\begin{split}
			LV\, (x)&= L\,[-\sum_z \rho^s(z)\, f(z)\,\tau(x,z) + \text{constant}]\\
   &= -\sum_z \rho^s(z) f(z) L \,\frS_{H}(x) 
			\\
			&= -\sum_z \rho^s(z) f(z) [g_z(x) - 1]
			\\
			&= f^s - f(x)
	\end{split}
 \end{equation}
which finishes the proof.  
 \end{proof}

\section{Graphical representations}\label{graph}
We next turn to a graphical representation of the solution of the Poisson equations \eqref{genp}--\eqref{prop}--\eqref{taup}. To be self-consistent and as a natural introduction to the Matrix Forest Theorem, if only for notation, we still remind the reader of the Kirchhoff formula; see  \cite{hb, intro} for more explanations and examples.\\

 \subsection{Kirchhoff formula}\label{seckir}
Remember that the Markov jump process, via its transition rates, has defined a  digraph  $G\,(\mathcal{V}(G), \mathcal{E}(G))$.\\
 $ H$ is a \textit{subgraph} of $ G$ if $\cal V( H)\subseteq \cal V( G)$ and $ \cal E(H)\subseteq \cal E( G)$. The subgraph is \textit{spanning} if $\cal V( H)=\cal V( G)$.\\
A \textit{tree} in $G$  is defined as a connected subgraph of $G$ that does not contain any cycles or loops.
The set of all spanning trees rooted in $x$ is denoted by   $\cal T_x$. 
\begin{proposition}\label{prokir}
    The following Kirchhoff formula gives the unique solution to the stationary Master Equation $ \rho^s L=0$ in terms of spanning trees, 
\begin{equation}\label{kir}
    \rho^s(x)=\frac{w(x)}{W}, \quad W:=\sum_y w(y),
\end{equation}
 with weights
\[
w(x)  = \sum_{T_x \in \cal T_x} w(T_x),\quad w(T_x)=\prod_{(u,u')\in T_x}k(u,u')
\]
where $\cal T_x$ is the set of all spanning trees rooted on $x$, and $T_x\in \cal T_x$ is a rooted spanning tree with arcs $(u,u')$ (directed edges).
\end{proposition}

The usual proof of the Kirchhoff formula \eqref{kir} proceeds via verification.  In Appendix \ref{rki} we add the heuristics of a more constructive proof,  \cite{kirana}.

\subsection{Graphical representation of the solution of the Poisson equation}

A (spanning) \textit{forest} in graph $G$ is a set of trees. For the moment we are only  concerned with two-trees, {\it i.e.}, forests that consist of exactly two trees with the understanding that a single vertex also counts as one tree.\\ 
Let $\cal F^{x\to y}$ denote  the set of all   rooted spanning forests with two trees, where $x$ and $y$ are in the same tree rooted at $y$. The second tree is rooted as well, and we include all possible orientations. The two-trees, elements of $\cal F^{x\to y}$, are denoted by $F^{x\to y}$. The weight of the set
\begin{equation}\label{ww}
w(x \to y) := w(\cal F^{x\to y}) = \sum_{F^{x\to y} \in \cal F^{x\to y} }w(F^{x\to y}), \quad w(F^{x\to y}) = \prod_{(u,u')\in F^{x\to y}} k(u,u')
\end{equation}
We write $\cal F^{x\to x} = \cal F^{x}$; see \fig\ref{exnotation}.
Moreover, let $\cal F^{x,\,y}$ be the set of all (spanning) forests consisting of two trees such that $x$ and $y$ are located in different trees and $y$ is a root; see \fig\ref{exnotation} where $\cal F^{x,\,y}=\cal F^{y}\setminus \cal F^{x\to y}$.\\ 

\begin{figure}[h!]
    \centering
    \includegraphics[scale=0.8]{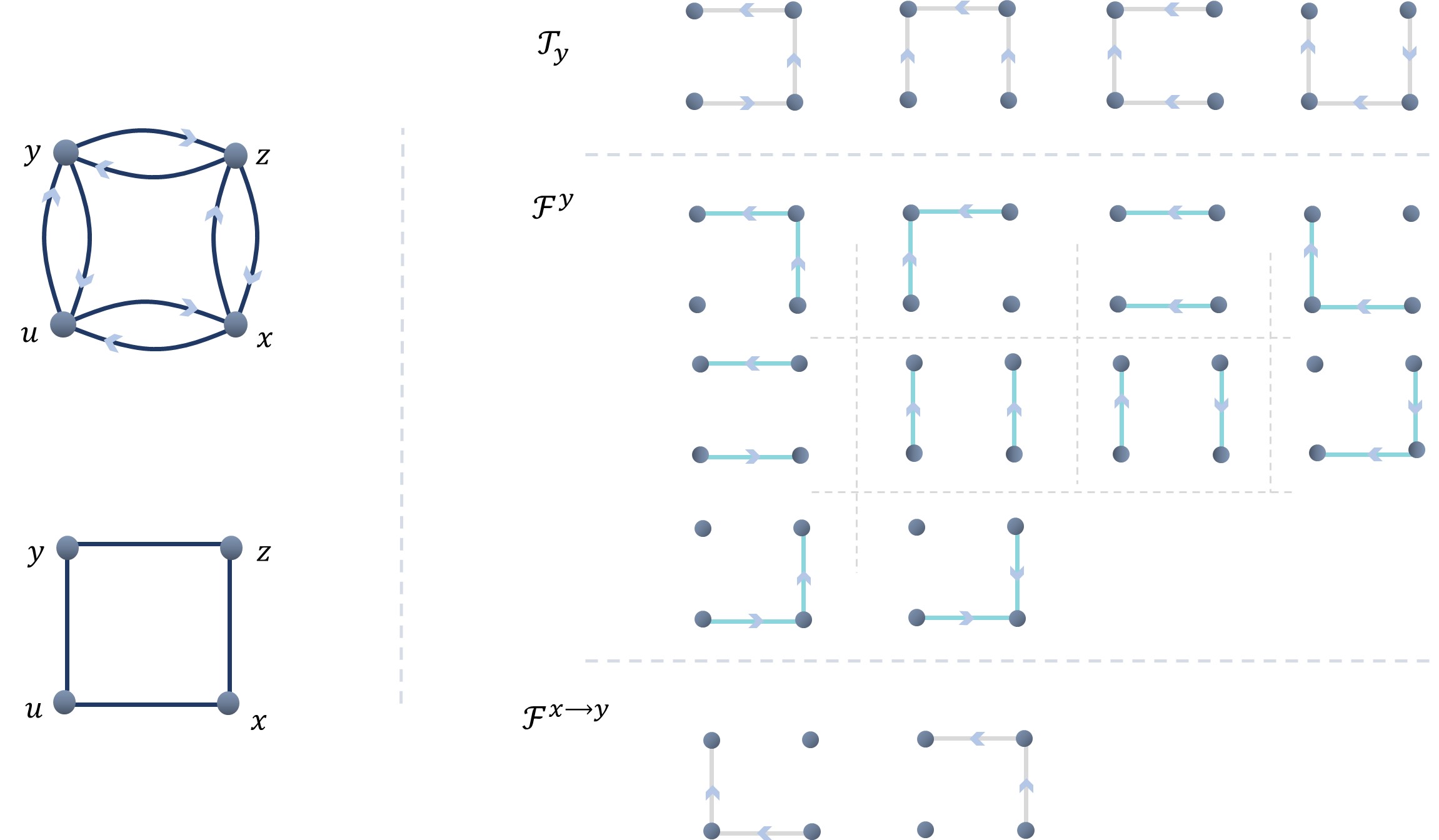}
    \caption{\small {A digraph with four vertices (top left) and its simple visualization (down left). Right: from the top onward, the elements of  the sets  $\cal T_y$, $\cal F^y$ and $\cal F^{x\to y}$ are shown.}}
    \label{exnotation}
\end{figure}

Recall the Poisson equation \eqref{prop} for the quasipotential $V$, where $f$ is a centered function on $K$. 
\begin{theorem}\label{th1}
The solution to \eqref{prop} for $\langle V\rangle^s=0$ is
 \begin{equation}\label{V}
    V(x)=\sum_y\frac{  w(x\to y)}{W}\,f(y)
\end{equation}
\end{theorem}
We proceed with the proof by direct verification.  An alternative proof, starting from a more general setup and using the Matrix Forest Theorem, is presented in Appendix \ref{broa}.
\begin{proof}[Proof of Theorem \ref{th1}]
\begin{align*}
    LV\, (x) &=\sum_yk(x,y)\,[V(y)-V(x)]\\
    &=\frac{1}{W}\,\sum_y k(x,y) \,\sum_z[w(y \to z)- w(x \to z)]\, f(z)
\end{align*}
where $w(y \to z)=w(\cal F^{y\to z} )$ and $w(x \to z)=w(\cal F^{x\to z})$ as in \eqref{ww}. The intersection of  $\cal F^{y\to z}$ and $\cal F^{x\to z}$ is the set of all forests where $x$ and $y$ are located in  the same tree. Here, to emphasize the positions $x$ and $y $, we use the notation $\cal F^{x,y\to z}$ to denote the set of all forests where $y$ is located in the tree rooted in $z$, and $x$ is in the other tree. Then,
\begin{align}\label{lv2}
    LV\, (x) &=\frac{1}{W}\,\sum_u \sum_y k(x,y) \,[w(\cal F^{x,y\to z})- w(\cal F^{y,x\to z})]\, f(z)
\end{align}
Fix $z\not=x$, take a forest  $F^{x,y\to z}\in \cal F^{x,y\to z}$ and  add the edge $(x,y)$ to that forest: one tree is rooted in $z$, and corresponding to the root of the second tree  there are two cases.   Consider the case that $x$ is not the root, then there exists an edge   $(x,y')$ on the tree. Remove the edge $(x,y')$ from the graph $(x,y) \cup F^{x,y\to z}$, a new forest is made. The new forest is in the set $F^{y',x\to z}$; see \fig\ref{f1}.

\begin{figure}[H]
     \centering
     \begin{subfigure}{0.49\textwidth}
         \centering
         \def\svgwidth{0.8\linewidth}        
         \includegraphics[scale = 0.7]{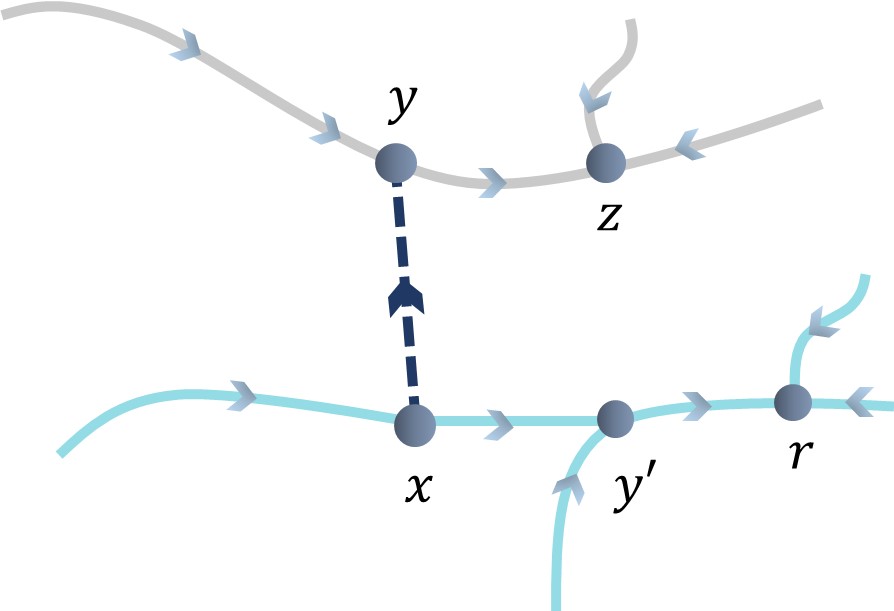}
         \caption{$(x,y)\cup F^{x,y\to z} $}
     \end{subfigure}
     \hfill
     \begin{subfigure}{0.49\textwidth}
         \centering
         \def\svgwidth{0.8\linewidth}        
\includegraphics[scale = 0.7]{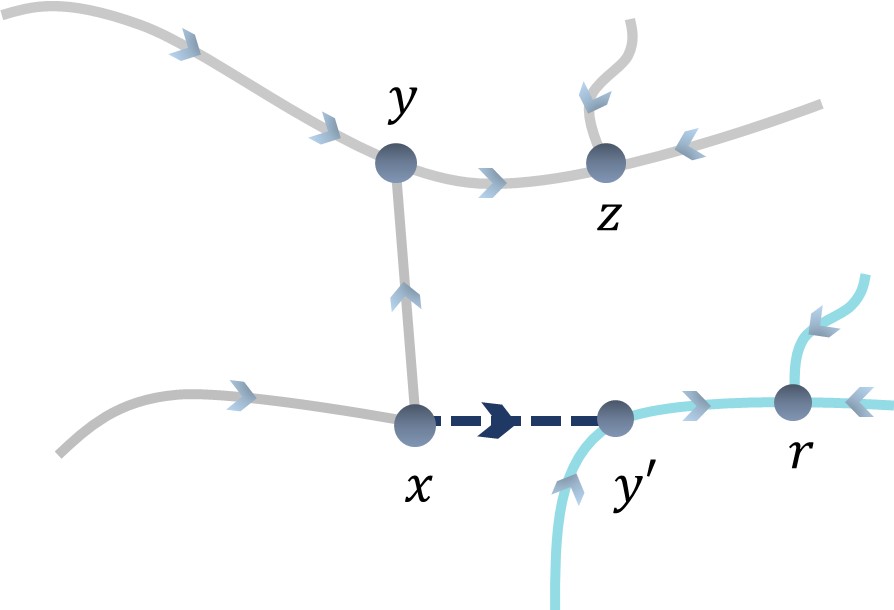}
         \caption{$(x,y')\cup F^{y',x\to z} $}
     \end{subfigure}
        \caption{\small{  Adding  $(x,y)$ to  $F^{x,y\to z}   $ and  the edge $(x,y')$ to  $F^{y',x\to z}$, make the same graphs. }}\label{f1}
\end{figure}

 The same scenario is true for the set $\cal F^{y,x\to z}$, see \fig\ref{f2}.
\begin{figure}[H]
     \centering
     \begin{subfigure}{0.49\textwidth}
         \centering
         \def\svgwidth{0.8\linewidth}        
         \includegraphics[scale = 0.7]{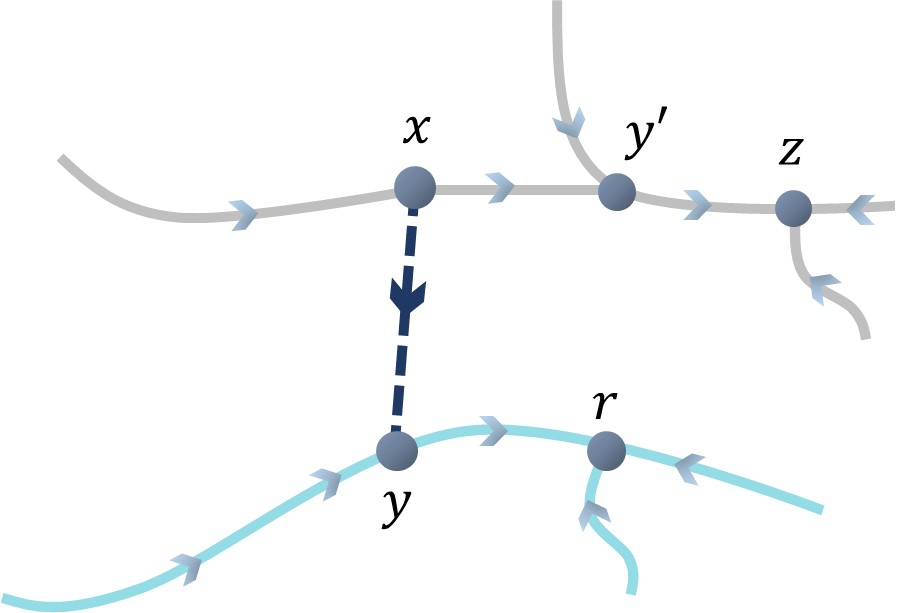}
         \caption{$(x,y)\cup F^{y,x\to z} $}
     \end{subfigure}
     \hfill
     \begin{subfigure}{0.49\textwidth}
         \centering
         \def\svgwidth{0.8\linewidth}        
\includegraphics[scale = 0.7]{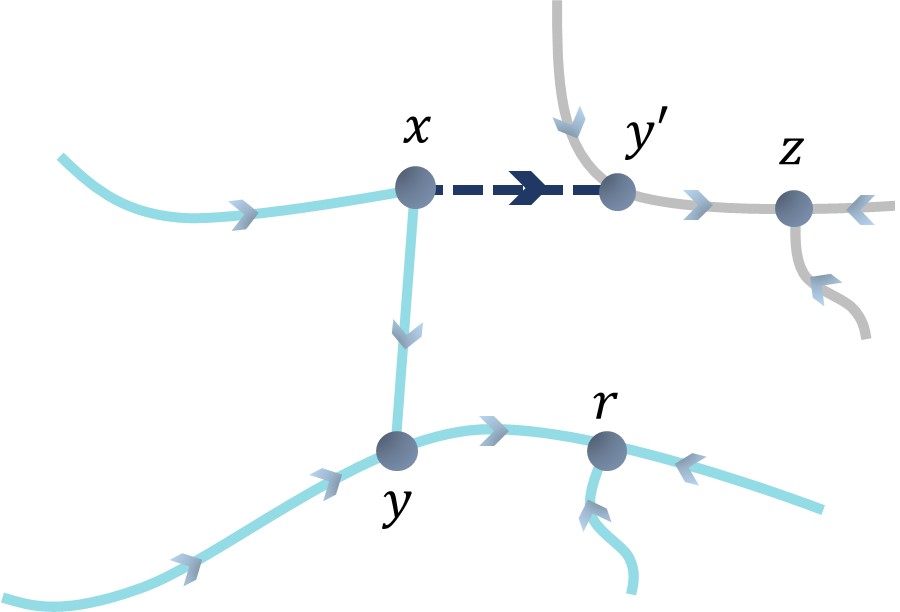}
         \caption{$(x,y')\cup F^{x,y'\to z} $}
     \end{subfigure}
        \caption{\small{  Adding  $(x,y)$ to  $F^{y,x\to z}   $ and  the edge $(x,y')$ to  $F^{x,y'\to zu}$, make the same graphs. }}\label{f2}
\end{figure}
Hence, when $x$ is not the root, \eqref{lv2} equals zero. \\
Consider a forest from  $\cal F^{x,y\to z} $. If  $x$ is  the root of its tree, then adding the edge $(x,y)$ connects two trees and the new graph is a spanning tree rooted in $z$; see \fig\ref{f3}. 
\begin{figure}[H]
    \centering
    \includegraphics[scale=0.7]{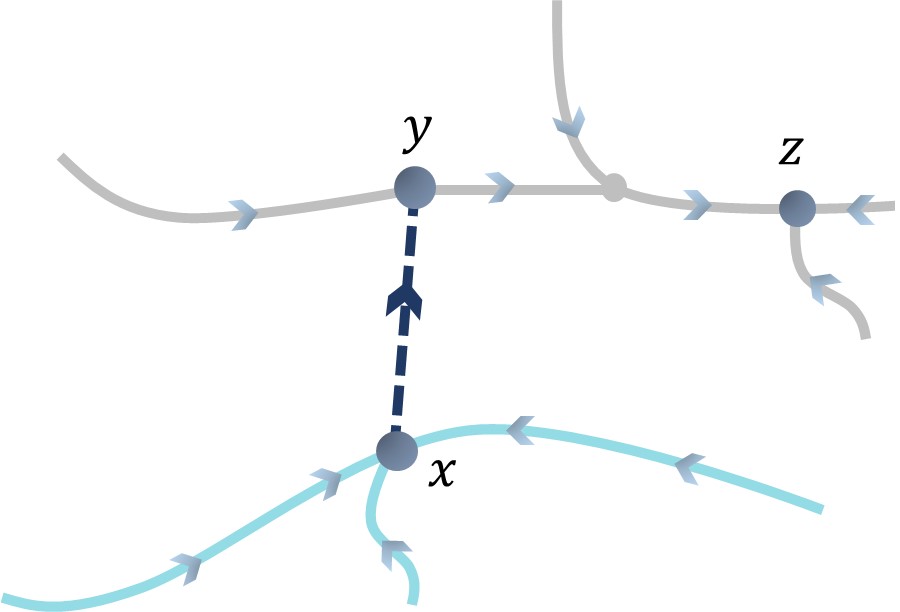}
    \caption{\small{ Adding $(x,y)$ to  the forest $F^{x,y\to z}$ where $x$ is the root makes a spanning tree rooted in $z$.}}
    \label{f3}
\end{figure}
 Put $z=x$, the set of $\cal F^{x,y\to x}$ is empty. Take a forest from the set $\cal F^{y,x\to x}$.  Adding the edge $(x,y)$ to the forest $F^{y,x\to x}$ again is a new spanning tree rooted on a vertex $r\not=x$.\\
Hence, finally,
 \begin{align*}
     LV\, (x)&=\frac{1}{W}[\sum_{z\not =x} w(z)f(z)-\sum_{r\not=x} w(r)f(x)]\\
     &=\frac{1}{W}\sum_{z\not =x} [w(z)f(z)-w(z)f(x)]\\
     &=\frac{1}{W}\sum_{z\not =x} [w(z)f(z)+w(x)f(x)-w(x)f(x)-w(z)f(x)]\\
     &=\frac{1}{W}\sum_{z} w(z)f(z) - \frac{f(x)}{W}\sum_z w(z)=\langle f\rangle-f(x)\\
     &=-f(x)
 \end{align*} 
where we used the Kirchhoff formula \eqref{kir} and the fact that $f$ is centered.
 \end{proof}

Next, recall the Poisson equation \eqref{taup} for the mean first-passage times.  There as well, we get a graphical representation.  The set $\cal F^{x,y}$ contains all two-trees with the restrictions:
\begin{itemize}
    \item vertices $x$ and $y$ are in separate trees;
    \item the tree that contains $y$ is rooted in $y$, while the second tree is rooted anywhere (containing $x$); 
    \item the weight $w(x,y) = w(\cal F^{x,y})$ of the set of such two-trees $F^{x,y}$ is
    \[
    w(x,y) = \sum_{F^{x,y} \in {\cal F}^{x,y}}
\prod_{(u,u')\in F^{x,y}} k(u,u')\]
\end{itemize}

\begin{theorem}\label{th2}
The mean first-passage time, solution to \eqref{taup},  is   
   \begin{equation}\label{fpt}
    \tau (x,z)=\frac{w(x,z)}{w(z)} \quad x\neq z;\qquad \tau(x,x)=0
\end{equation}
where $w(z)$ appeared in \eqref{kir}.
\end{theorem}
We again give the proof of Theorem \ref{th2} by verification.
\begin{proof}[Proof of Theorem \ref{th2}]
We need to show 
\[ 
\sum_y k(x,y) [\frac{w(y,z)}{w(z)} - \frac{w(x,z)}{w(z)}] = -1
\]
or
\begin{equation}\label{tau1}
    \sum_y k(x,y)[w(x,z) - w(y,z)]=w(z).
\end{equation}
where $w(x,z)=w(\cal F^{x,\,z})$ and $w(y,z)=w(\cal F^{y,\,z})$,
\begin{align*}
    w(x,z) - w(y,z)&=\sum_{F^{x,\,z}\in \cal F^{x,\,z}}w(F^{x,\,z})-\sum_{F^{y,\,z}\in \cal F^{y,\,z}}w(F^{y,\,z})
\end{align*}
The intersection of the sets $\cal F^{x,\,z}$ and $\cal F^{y,\,z}$  is a set of  forests where  $x$ and $y$ are located in the same tree. We only consider the sets of forests where $x$ and $y$ are in  different trees,
\begin{align*}
    w(x,z) - w(y,z)&=  w(\cal F^{x,y \to z})-w(\cal F^{y,x \to z})
\end{align*}
 $\cal F^{x,y \to z}$ denotes the set of all forests, where $x$ and $z$ are in the different trees and $y$ is  in the same tree with $z$.\\
Take a forest $F^{x,y \to z}$ from the set $\cal F^{x,y \to z}$.  We consider two cases based on whether $x$ is a root or not.  If $x$ is not the root in forest  $F^{x,y \to z}$, then $x$ has  a neighbor such as $y'$.  Add the edge $(x,y)$ where $y$ is located in the same tree as $z$: a directed graphical object is created; see \fig \ref{f1}~(a). The new graphical object is  created also by adding edge $(x,y')$ to a forest from the set $\cal F^{y',x \to z}$, where $x$ and $z$ are located in the same tree in  the forest and $y'$ is in another tree; see \fig\ref{f1}~(b). The same graphical objects have equal weights.\\
If $x$ is a root, adding the edge $(x,y)$ connects two trees and creates a new spanning tree rooted at $z$; see \fig\ref{f3}. Consider a spanning tree rooted at vertex $z$, where each of the other vertices has an outgoing edge.  Remove the edge that goes out from $x$: a forest in the set $\cal F^{x,\,z}$ is created.  That ends the proof.  
\end{proof}
\begin{corollary}
Independently of $x$,
 \begin{equation}\label{tas}
\sum_y \rho^s(y) \tau(x,y) = \frac{W_2}{W}
\end{equation}
where $W_2$ is the total weight of all spanning two-tree forests with both trees subsequently oriented toward all its states.    
\end{corollary}
\begin{proof}[Proof of \eqref{tas}]
 \begin{align*}
     \sum_y \rho^s(y) \tau(x,y) &= \sum_y\frac{w(x,y)}{w(y)} \, \frac{w(y)}{W}=\frac{1}{W}\sum_yw(x,y)
 \end{align*}
 and
 \[
 W_2 =  \sum_yw(x,y)
 \]
for every $y$,  there is a forest  $F^{x,y}\in \cal F^{x,y} $ where  $x$ is the root and for all $y \not =x$  also $y$ is the root. So then $W_2$ is the weight of the set of all forests with two rooted trees and it is independent of $x$.
\end{proof}

\section{Bounds}\label{bou}
One important issue for possible applications of the Poisson equation is to get bounds on its solution.  One could imagine in fact a parameterized family of Poisson equations and the issue becomes to get bounds that are uniform in that parameterization.  The present section adds such bounds, as a consequence of the previous sections.

\subsection{Via mean first-passage times}
Consider the centered solution $V$ of the Poisson equation \eqref{prop}, with $||f|| := \max_{z\in K}|f(z) -\langle f\rangle^s|$.
Fix any two states $x,y \in K$, and put
\begin{equation}\label{vv}
  \tilde{v}(x,y) :=  \left\langle\int_0^{T_{K\setminus \{y\}}}\id t\, [f(X_t) - \langle f\rangle^s]\,|\,X_0=x\right\rangle
\end{equation}
as the expected first-passage accumulation for the centered observable
$ f - \langle f \rangle^s$.
\begin{theorem}\label{th3}
 We have
    \begin{equation}\label{eap}
        V(x) - V(y) =  \tilde{v}(x,y)
    \end{equation}
 and the bound
    \begin{equation}\label{bb}
    |   V(x) - V(y)| \leq ||f||\min\{\tau(x,y),\tau(y,x)\}
    \end{equation}
\end{theorem}
\begin{proof}[Proof of Theorem \ref{th3}]
To prove \eqref{eap}, we note that the centered solution $V(x)$ to \eqref{prop} can be written by fixing state $y$, and decomposing as
 \begin{equation}
 	\begin{split}
 		V(x) &= \lim_{t \to \infty} 
 		\Bigl\langle \int_0^{T_{K \setminus \{y\}}} f(X(t'))\,\id t'
 		+ \int_{T_{K \setminus \{y\}}}^{t} f(X(t'))\,\id t' - \langle f\rangle^s t
 		\Bigr\rangle_x		
 		\\
 		&=\Bigl \langle\phi_{K \setminus \{y\}}(x) - \langle f\rangle^s\,\langle  T_{K \setminus \{y\}} \rangle_x + \lim_{t \to \infty}
 		\Bigl\langle \int_0^t [f(X(t')) - \langle f\rangle^s]\,\id t' \Bigr\rangle_y  
 		\\
 		&= V_{K \setminus \{y\}}(x) - \langle f\rangle^s \frS_{K \setminus \{y\}}(x) 
 		+ V(y)
 	\end{split}
 \end{equation}
 where in the limit of the second integral we have used the exponential tightness of $T_{K \setminus \{y\}}$. 
 As before, $\frS_{K \setminus \{y\}}(x) = \tau(x,y)$ is the mean first-passage time to reach $y$ from $x$, and 
 $V_{K \setminus \{y\}}(x) =: v(x,y)$
 is the expected accumulation for $f$ up to the first passage a $y$, as in Section \ref{sae}. Therefore, we have the relation
 \begin{equation}\label{quasipotential-accumulation}
 	\begin{split}
 		V(x) - V(y) &= v(x,y) - \langle f\rangle^s \,\tau(x,y)		
 		\\
 		&= \tilde{v}(x,y)
 	\end{split}
 \end{equation}
with $\tilde{v}(x,y)$ defined in \eqref{vv}.\\
Finally, from  the very definition of $\tilde{v}(x,y)$
\[
|\tilde{v}(x,y)| \leq ||f||\,\tau(x,y)
\]
 Moreover, from \eqref{eap}, $\tilde{v}(x,y)= -\tilde{v}(y,x)$ is antisymmetric, which yields the bound \eqref{bb}.
 \end{proof}

Note that since $\langle V\rangle^s=0$, there are always states $x_0,x_1\in K$ with $V(x_0) \leq 0\leq V(x_1)$.  Therefore, for every $x\in K$,
\[
V(x) \leq V(x)-V(x_0) \leq \sum_{x_0\rightarrow x}|V(x_i) - V(x_{i+1})| 
\]
where the sum is over an arbitrary path in $K$ connecting $x_0$ and $x$.  Similarly,
\[
V(x) \geq V(x)-V(x_1) \geq -\sum_{x_1\rightarrow x}|V(x_i) - V(x_{i+1})| 
\]
As a consequence, bounds on the solution $V$ follow from bounds on the differences, as provided in \eqref{bb}.\\

The bound \eqref{bb} is not always optimal.  
It still can be improved as follows.
\begin{theorem}\label{th4}
Suppose there exists a set $\cal D$ so that  $f(x) = LE(x) + h(x)$ where $h=h_E$ depends on $E$ and $h(x)=0$ for $x\notin \cal D$. Then, the solution of the Poisson equation \eqref{prop}
$V$  satisfies
 \begin{equation}\label{bf}
    |V(x)-V(y)| \leq |E(x)-E(y)|  + ||h_E||\sum_{z\in {\cal D}} \rho^s(z) |\tau(x,z)-\tau(y,z)|
\end{equation}
\end{theorem}
That bound requires to control the accessibility of the set $\cal D$ only.  In fact, since there is a freedom in choosing $(E,h)$, that can be used to optimize the bound \eqref{bf}.
We observe that in the case where $f\equiv LE$, the bound is optimal because then $V=E + \text{constant}$.
\begin{proof}[Proof of Theorem \ref{th4}]
The Poisson equation now becomes
 \[
 L(V+E) + h =0, \quad \text{ with }  h=0 \text{ on } K\setminus \cal D
 \]
Applying \eqref{gr}, we get
\begin{equation}\label{gra}
    V(x) + E(x) = -\sum_{z\in \cal D} \rho^s(z)\, h_E(z)\,\tau(x,z) + \text{constant}
\end{equation}
and the bound \eqref{bf} follows directly.
\end{proof}

\subsection{Via graphical representation}
We recall the setup where a  digraph with $n$ vertices is obtained from a Markov jump process with transition rates $k(u,u')$.   We define
\begin{equation}\label{defmax}
  ||f||:=\max_y |f(y)|, \quad ||k|| :=\max_{(u,u')} k(u,u')
\end{equation}
\begin{theorem}\label{th5}
For all $x$,
 \begin{equation}\label{bound}
 |V(x)| \leq  \frac{n\,||k||^{n-2}}{W} \,||f||
\end{equation}
\end{theorem}
\begin{proof}[Proof of Theorem \ref{th5}]
We apply \eqref{V} of Theorem \ref{th1}.  The product over the weights over edges is obviously bounded, and the weight of any forest in the graph is bounded as well. In other words,  for any $x$ and $y$ in the graph, $w(x\to y)$ is bounded. \\
Using \eqref{defmax},
\[
w(x\to y)f(y)\leq \,||f||\,||k||^{n-2}
\]
and the result follows from \eqref{V}.
\end{proof}

 As an application of \eqref{bound}, we can suppose that the rates $k(x,y)= k_\lambda(x,y)$ depend on a real-valued parameter $\lambda$.  For simplicity, we consider the limit $\lambda\rightarrow \infty$  (taking for $\lambda$ the inverse temperature is a relevant example from statistical mechanics).  
Observe that there is a general lower bound on $W$, {\it i.e.}, $W \geq \max_{T,x} w(T_x)$, the largest weight of all spanning trees. Hence, if there exists a spanning tree with $w(T_x) > 0$ uniformly in $\lambda$, then $W=W_\lambda > 0$ is uniformly bounded, and \eqref{bound} gives a uniform bound on $V$.\\  The bound \eqref{bound} can for instance be used in an argument extending the Third Law of Thermodynamics to nonequilibrium systems; see \cite{jchemphys}  where the source function $f$ is the expected heat flux.  The quasipotential $V$ in \eqref{vint} is then the time-integral of the difference between the instantaneous heat flux and the stationary heat flux.

\section{Conclusions}
The Poisson equation is ubiquitous in mathematical physics.  We have considered a setup where the linear operator is the backward generator of a Markov jump process and the functions are defined on the possible (finite number of) states. We have related the solution of that discrete Poisson equation to the formalism of mean first-passage times, and we have given graphical representations that allow precise estimates---both for the accumulations before first hitting a particular state when starting from another one, and for the quasipotentials measuring the accumulated excess during relaxation to stationarity.  The results do not assume reversibility, and are ready to be used in extensions of potential theory when applied to problems in steady-nonequilibrium statistical mechanics.  Obvious targets are reaction rate theory for driven or active systems, and metastability around \emph{non}equilibrium steady conditions.

\appendix
\section{On proving the Kirchhoff formula}\label{rki}
The Kirchhoff formula can be seen as a consequence  of the Matrix Tree Theorem, \cite{shubertkir,tutte1948dissection,kirphi}.  However, there also exists a more probabilistic approach. The ideas  is illustrated by the algorithm  in \fig \ref{kirtree}, where a tree $T'_x$ is constructed by removing the edge $(x,y')$ from the tree $T_y$, and then adding the edge $(y,x)$. That algorithm is used in \cite{kirana} to construct  a reversible Markov chain $Y_n$ on an ensemble of trees, where the states are the rooted spanning trees and jumps are possible if and only if the corresponding trees are connected by the above-mentioned algorithm. Let $Q_{ab} = q(a,b)$ be the stochastic matrix of $Y_n$, and put $\pi$ the stationary distribution of $Q$. For example, we can assign states $a$ and $b$, respectively, to the trees $T_y$ and $T'_x$ in  \fig \ref{kirtree}, and take transition rates $q(a,b)$. It is shown that putting $q(a,b)=k(y,x)$, $q(b,a)=k(x,y')$ produces detailed balance $\pi_a q(a,b)=\pi_b q(b,a)$. 
We add a sketch of the proof.\\

It suffices to show that the Kirchhoff formula \eqref{kir} satisfies the stationary Master equation,
     \[
     \sum_y w(y)k(y,x)-w(x) k(x,y)=0
     \]
Take a  spanning tree rooted in $y$,  $T_y \in \cal T_y$.  Let $x$ be a vertex located on the tree $T_y$. In this case, $x$ is not the root and there exists a vertex $y'$ (which may be equal to $y$) and an edge $(x,y')$ which goes out from $x$ and   connects $x$ to the tree. By removing the edge $(x,y')$ from the  tree $T_y$ and adding the edge $(y,x)$, a new spanning tree rooted in $x$ is created; see \fig \ref{kirtree}.
\begin{figure}[H]
    \centering
    \includegraphics[scale=0.8]{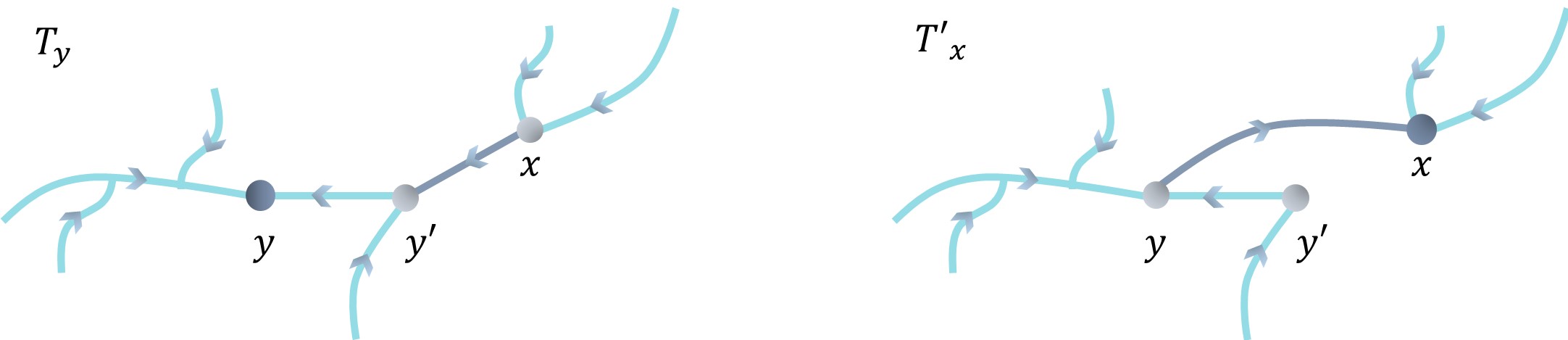}
    \caption{\small{Removing the edge $(x,y')$ from the spanning tree $T_y$ and adding the edge $(y,x)$ creates a spanning tree rooted in $x$.}}
    \label{kirtree}
\end{figure}
It turns out that $\dfrac{w(T_y)}{k(x,y')}=\dfrac{w(T'_x)}{k(y,x)}$. With the same scenario, from every spanning tree rooted in $y$ we can make a new spanning tree rooted in $x$. That algorithm  makes a  spanning tree rooted in $y$ from a spanning tree rooted in $x$. Therefore, 
\begin{align*}
 \sum_y w(y)k(y,x)=\sum_y \sum_{T_y\in \cal T_y} w(T_y)k(y,x)=\sum_{y'} \sum_{T'_x\in \cal T_x}w(T'_x)k(x,y')=\sum_{y'} w(x)k(x,y')
\end{align*}
which  ends  the proof. 

 \section{Matrix Forest Theorem}\label{broa}
The present Appendix gives the broader mathematical context of the graphical representations.
 \subsection{Laplacian matrix and the backward generator}\label{Laplacian}

Consider  a  self edge-free directed graph $\cal G(\cal V(\cal G), \cal E(\cal G))$, with $n$ vertices and $\frm$ edges. The \textit{adjacency matrix}  $A(\cal G)$ is a $n\times n$ matrix with elements $A(\cal G)_{x,y}=w_{x,y}$.\\
The  out-weight of a vertex $x\in \cal V(G)$ is the sum $\sum_y w_{x,y}$, and we define the  $n\times n$ diagonal matrix $D(G)$ with $D_{x,y}=0$ if $x\not =y $ and otherwise $D_{x,x}=\sum_y w_{x,y}$.\\
The \textit{Laplacian matrix} for a directed graph $\cal G$ is, 
\begin{equation}\label{lapla}
    \cal L(\cal G)=D(\cal G)-A(\cal G).
\end{equation}
Clearly, $\cal L(\cal G)$ is not always symmetric, and the backward generator of the Markov jump process (as we had it in Section \ref{mar}) $L = - \cal L $  is minus the Laplacian matrix of the underlining graph when putting $w_{x,y}=k(x,y)$.  Therefore, all properties of the Laplacian of a weighted digraph, as described in section \ref{Laplacian}, also hold for the backward generator $L$, as described in Section \ref{mar}.

The Laplacian matrix of a (directed) graph has several important properties that have been extensively studied in the literature \cite{chebo2002,adi}.

\subsection{Pseudo-inverses of the Laplacian matrix}\label{dg}
 The Laplacian is not invertible, and different pseudo/generalized-inverses can be defined, \textit{e.g.} group inverse, Drazin inverse, Moore–Penrose inverse and resolvent inverse. Depending on the characteristics of the graph,  the different inverses of the Laplacian  can  be equal or not; see also \cite{drazin}.

As a reminder, for an arbitrary square matrix $A$, its \textit{Drazin inverse} is denoted by $A^D$ and is the unique matrix $X$ satisfying the following equations
\begin{equation}\label{D}
    A^{\nu+1}\, X=X\, A^\nu, \quad X\,A\,X=X,\quad A\,X=X\,A
\end{equation}
where $\nu=\text{ind}\,A$ (remember $\text{ind}\,A$ is the smallest non negative integer number $b$ such that $\text{rank} (A^{b+1})=\text{rank}(A^b)$). If $\nu=0$, then $A^D=A^{-1}$; if $\nu \leq 1$, then $A^D$ is referred to  the \textit{group inverse} and it is denoted by $A^\#$ which is a unique matrix such as $X$ satisfying the following equations
\begin{equation}\label{G}
    A\, X\,A=A, \quad X\,A\,X=X,\quad A\,X=X\,A
\end{equation}
As a consequence, when the index of a matrix is one its Drazin inverse and group inverse are the same; see \ref{dg}.\\

We refer to \cite{chebo2006max} for the Matrix Forest Theorem  for a Laplacian matrix.  We briefly recall that result.\\

Consider a weighted digraph  $\cal G$, where $n$ is the number of vertices. $I$ is the identity matrix, $\cal F^{x \to y}_m$   is the set of all forests with $m$ edges such  that $x$  and $y$  are in the same tree and $y$ is always a root, $\cal F_m$ is the union of all $\cal F^{ x}_m$, and $m$ is the number of edges making the forests.  For example, if $m=n-1$ the forests are made by one spanning tree with $n-1$ edges, and when $m=n-2$ the forest has  two trees. $w(\cal F^{x \to y}_m)$ is the weight of the set $\cal F^{x \to y}_m$. $\gamma$ is the dimension of the forest in $\cal G$. Forest dimension is the minimum number of rooted trees that a spanning rooted forest can have in a directed graph.  Note that when the graph is strongly connected  $\gamma=1$.

Consider a weighted digraph $\cal G$ with the Laplacian matrix $\cal L$.  From \cite{chebo2006matrixforest}, 
\begin{theorem}
    For any $\alpha >0$, the matrix $(I+\alpha\, \cal L (\cal G))^{-1}$ has the graphical representation 
\begin{align}\label{aDla}
\bigg( \dfrac{1}{I + \alpha\, \cal  L}\bigg)_{x,y} = \frac{\sum_{m=0}^{n-\gamma} \alpha^m w(\cal F^{x \to y}_m)}{\sum_{m=0}^{n-\gamma}\alpha ^m w(\cal F_m)}
\end{align}
\end{theorem}
For the specific case when  $ \gamma=1$, 
\begin{align}\label{apDlaga1}
\bigg( \dfrac{1}{I + \alpha\, \cal  L}\bigg)_{x,y} =\frac{\sum_{m=0}^{n-2} \alpha^m w(\cal F^{x \to y}_m)}{\sum_{m=0}^{n-1}\alpha ^m w(\cal F_m)} + \frac{ \alpha^{n-1} w(\cal T_y)}{\sum_{m=0}^{n-1}\alpha ^m w(\cal F_m)}
\end{align}
where we have used $\cal F^{x \to y}_{n-1}=\cal T_y$, {\it i.e.}, the forest with $n-1$ edges is a spanning tree. $w(\cal T_y)$ is the weight of the set $\cal T_y$, $w(\cal T_y)=w(y).$\\

From \cite{chebo2002, chebofpt} the graphical representation for  the group inverse of $\cal L^{\#}$ is,
\begin{theorem}\label{lalpgroup}
\begin{equation}\label{apGla}
  \cal L^\#_{x,y}=\frac{1}{w(\cal F_{n-\gamma})}\big(w(\cal F_{n-\gamma-1}^{x\to y})-w(\cal F_{n-\gamma-1})\frac{w(\cal F_{n-\gamma}^{x\to y})}{w(\cal F_{n-\gamma})}\big)
\end{equation}
\end{theorem}
 If $\gamma=1$
\begin{align}\label{aGlaga1}
    \cal L^\#_{x,y}&=\frac{1}{w(\cal F_{n-1})}\big(w(\cal F^{x\to y})-w(\cal F)\, \frac{w(\cal F_{n-1}^{x\to y})}{w(\cal F_{n-1})}\big)\notag\\
    &=\frac{1}{\sum_y w(\cal T_y)}\big(w(\cal F^{x\to y})-w(\cal F)\,\frac{w(\cal T_y)}{\sum_y w(\cal T_y)}\big)\\
    &=\frac{1}{W}\big(w(\cal F^{x\to y})-w(\cal F)\,\rho^s(y)\big)
\end{align}
where we have used that the set of all rooted spanning forests with $n-1$ edges is indeed the set of all rooted spanning trees. $\cal F$  denotes the set of all rooted spanning forests consisting of two trees. In the last line we used the Kirchhoff formula.\\

The group inverse of the backward generator is denoted by  $ L^{\#}$ and from Theorem \ref{lalpgroup} the graphical representation is 
\begin{align}\label{Glaga1}
  L^\#_{x,y}=-\frac{w(\cal F^{x\to y})\,}{W}+ \rho(y)\, \,\frac{w( \cal F)}{W}
\end{align}

\subsection{Another proof of Theorem \ref{th1}}
We can solve  the equation $\eqref{prop}$ by utilizing a graphical representation of  the pseudoinverse of $L$, which can be obtained through the application of the Matrix Forest Theorem. 
\begin{proof}
Let us assume that $f$ is centered, meaning $\sum_yf(y)\, \rho^s(y)=0$. 
Use the graphical representation  \eqref{aDla},
\begin{align}
  V(x)&=\lim_{\alpha \to \infty}\sum_y   \bigg( \dfrac{\alpha}{I - \alpha\,  L}\bigg)_{x,y} \,f(y)\notag\\
  &=\lim_{\alpha \to \infty} \sum_y  \frac{\sum_{m=0}^{n-2} \alpha^{m+1} w(\cal F^{x \to y}_m)}{\sum_{m=0}^{n-1}\alpha ^m w(\cal F_m)} +\lim_{\alpha \to \infty} \sum_y\frac{ \alpha^{n} w(\cal T_y)}{\sum_{m=0}^{n-1}\alpha ^{m} w(\cal F_m)}f(y)\notag\\
  &=\sum_y \frac{  w(\cal F^{x \to y})}{W}\,f(y)
\end{align}
in the second line we use the Kirchhoff formula where $w(\cal F_{n-1})=W$ and apply the centered property of $f$ and the proof is finished.
 \end{proof}

As understood in \cite{drazin}, the resolvent inverse and the Drazin inverse of the backward generator on a centered function $f$ give the same solution to  the Poisson equation. The reason is that the index of the backward generator $L$ equals one; see \cite{chebo2006}. We check that explicitly in their graphical representations.
\begin{align*}
  V= - L^\#f\, (x)&=\sum_y\frac{w(\cal F^{x\to y})\, f(y)}{W}- \frac{w( \cal F)}{W} \sum_y\rho^s(y)\, f(y)\\
    &=\sum_y\,\frac{w(\cal F^{x\to y})}{W}\, f(y)
\end{align*}
which is equal to the solution of the Poisson equation via the resolvent inverse; see $\eqref {V}$.  Hence, if   $f$ is centered, then the resolvent inverse, Drazin inverse and group inverse of $L$ are all equal giving $V\,=-\,L^{-1} \, f$.
\subsection{Another proof of Theorem \ref{th2}}
The equation in \eqref{taup} can be solved by the graphical representation of the group inverse of the backward generator.
\begin{theorem}
    The solution of $L\tau +1=0$ is   
  \begin{equation}\label{mftg}
   \tau=(L^\#-{\mathbb 1} \, L^\#_{dg})D 
  \end{equation}
where $\mathbb{1}$ is a $n \times n$ matrix such that  all elements are $1$ and  $D$ is a diagonal $n \times n$ matrix such that $D_{xx}= \frac{1}{\rho(x)}$. $L^\#_{dg}$ is a matrix made by putting all the entries outside of the main diagonal of  $L^\#$ equal to  zero. 
\end{theorem}

We start by
\[
L\, \tau = L \,(L^\#-\mathbb{1} \, L^\#_{dg})\,D =L\,L^\# D\]
where $L\mathbb{1} = 0$ is the  $n \times n$ matrix with all entries equal to zero. Furthermore,
\begin{align*}
  (L\,L^\#\, D)_{xz}&= (L\,L^\#\,) _{xz}\frac{1}{\rho^s(z)}\\
  &=\frac{1}{\rho^s(z)}\sum_y L_{xy} L^\#_{yz} =\frac{1}{\rho^s(z)}(\sum_{y\not= x}k(x,y)L^\#_{yz}-L^\#_{xz}\sum_yk(x,y))\\
 &= \frac{1}{\rho^s(z)}\sum_y k(x,y)(L^\#_{yz}-L^\#_{xz})
 \end{align*}
 \begin{align}\label{inter}
(L\,L^\#\, D)_{xz}&=-\frac{1}{\rho^s(z)}\sum_y k(x,y)\frac{1}{W}\bigg(\big(w(y\to z)-\rho^s(z)\sum_{z,y}w(y\to z)\big)\notag\\
&\qquad \qquad \qquad \qquad \qquad-\big(w(x\to z)-\rho^s(z)\sum_{x,z}w(x\to z)\big)\bigg)\\
&=-\frac{1}{w(z)}\sum_y k(x,y)\bigg(w(y\to z)-w(x\to z)\bigg)
\end{align}
Fix $x$ and $z$, consider two sets $\cal F^{y\to z}$ and $\cal F^{x\to z}$. For the fix $y$,  the intersection of them is a set of all forests such that $x$ and $y$ are in the same tree rooted at $z$. As the result, in \eqref{inter}, we only consider  the component of  the set $\cal F^{y\to z}$ where   $x$ is  not in the same tree with $y$ and $z$, and in the set  $\cal F^{x\to z}$ only the components where $z$ is not in a same tree with $x$ and $z$. Hence, 
\[
w(y\to z)-w(x\to z)=w(\cal F^{x,\,y\to z})-w(\cal F^{y,\,x\to z})
\]
where $ \cal F^{x,\,y\to z}$ is set of all forests such that $z$ is on the tree rooted in $y$ and $x$ is located on the another tree.\\
Take a forest from the set $\cal F^{x,\, y\to z}$.  If $x$ is the root, then adding edge the $(x,y)$ connects two trees such that a new spanning tree rooted in $z$ is created, see \fig \ref{f3}. Now consider a spanning tree rooted at $z$, by removing the edge goes out from $x$  a forest in the set $\cal F^{x,\,\,z}$ is created. \\
Take a forest from the set $F^{x,y\to z}\in \cal F^{x,y\to z}$ such that $x$ is not the root. If $x$ is not the root, then it has  a neighbor such as $u$. Add the edge $(x,y)$ to $F^{x,y\to z}$, a directed graphical object is created. The new graphical object is  created also by adding edge $(x,y')$ to a forest from the set $\cal F^{y',x\to z}$, see \fig \ref{f1}. The same graphical objects have equal weights and as a result 
   \[ \sum_y k(x,y)\big( w(\cal F^{x,\,y\to z})-w(\cal F^{y,\,x\to z})\big)=w(z)\]
   and 
\begin{equation}
(L\,L^\#\, D)_{xz}=-1 \qquad z\not =x
\end{equation}
It follows $L \tau=-1$ for every $z\not =x$, which ends the proof of \eqref{mftg}.\\
We can get the graphical representation \eqref{fpt} from  \eqref{mftg}:
\[\tau (x,z)=\begin{cases}
        0 &x=z\\
        \frac{L^\#_{zz}-L^\#_{xz}}{\rho^s(z)}& x\not=z
    \end{cases}\]
   where 
   \[L^\#_{zz}-L^\#_{xz}= \frac{1}{W}\big(w(z \to z)-w(x\to z)\big)=\frac{w(x,z)}{W}\]
   Hence. we get indeed
   \begin{equation*}
    \tau (x,z)=\frac{w(x,z)}{w(z)} \quad x\neq z;\qquad \tau (x,x)=0
\end{equation*}

\bibliographystyle{unsrt}  
\bibliography{chr}
\onecolumngrid

\end{document}